\documentclass{amsart}
\usepackage{amsmath,amsthm,amsfonts, amssymb,amscd}
\usepackage[all]{xy}

\newtheorem{theorem}{Theorem}[section]
\newtheorem{definition}[theorem]{Definition}
\newtheorem{example}[theorem]{Example}
\newtheorem{corollary}[theorem]{Corollary}
\newtheorem{proposition}[theorem]{Proposition}
\newtheorem{problem}[theorem]{Problem}
\newcommand{\minusre}{\hspace{0.3em}\raisebox{0.3ex}{\sl \tiny /}\hspace{0.3em}}
\newcommand{\minusli}{\hspace{0.3em}\raisebox{0.3ex}{\sl \tiny $\setminus $}\hspace{0.3em}}
\newcommand{\lex}{\,\overrightarrow{\times}\,}

\newcommand{\Ker}{\mbox{\rm Ker}}

\newcommand{\Infinit}{\mbox{\rm Infinit}}

\def\rd{\mathord{/}}

\begin{document}
\title[Lexicographic Pseudo MV-algebras]{Lexicographic Pseudo MV-algebras}
\author[Anatolij Dvure\v{c}enskij]{Anatolij Dvure\v censkij$^{1,2}$}
\date{}%
\maketitle

\begin{center}  \footnote{


This work was supported by  the Slovak Research and Development Agency under contract APVV-0178-11,  grant VEGA No. 2/0059/12 SAV, and GA\v{C}R 15-15286S.
 }
Mathematical Institute,  Slovak Academy of Sciences\\
\v Stef\'anikova 49, SK-814 73 Bratislava, Slovakia\\
$^2$ Depart. Algebra  Geom.,  Palack\'{y} University\\
17. listopadu 12, CZ-771 46 Olomouc, Czech Republic\\

E-mail: {\tt dvurecen@mat.savba.sk}
\end{center}

\begin{abstract}
A lexicographic pseudo MV-algebra is an algebra that is isomorphic to an interval in the lexicographic product of a linear unital group with an arbitrary $\ell$-group. We present conditions when a pseudo MV-algebra is lexicographic. We show that a key condition is the existence of a lexicographic ideal, or equivalently, a case when the algebra can be split into comparable slices indexed by elements of the interval $[0,u]$ of some unital linearly ordered group $(H,u)$. Finally, we show that fixing $(H,u)$, the category of $(H,u)$-lexicographic pseudo MV-algebras is categorically equivalent to the category of $\ell$-groups.
\end{abstract}

\section{Introduction}

MV-algebras are the algebraic counterpart of the infinite-valued \L ukasiewicz sentential calculus introduced by Chang in \cite{Cha}.
Perfect MV-algebras were characterized as MV-algebras where each element is either infinitesimal or co-infinitesimal. Therefore, they have no parallels in the realm of Boolean algebras because perfect MV-algebras are not semisimple. The logic of perfect pseudo MV-algebras has a counterpart in the  Lindenbaum algebra of the first order \L ukasiewicz logic which is not semisimple, because the valid but unprovable formulas are precisely the formulas that correspond to co-infinitesimal elements of the Lindenbaum algebra, see e.g. \cite{DiGr}. Therefore, the study of perfect MV-algebras is tightly connected with this important phenomenon of the first order \L ukasiewicz logic.

Recently, two equivalent non-commutative generalizations of MV-algebras, called pseudo MV-algebras in \cite{GeIo} or GMV-algebras in \cite{Rac}, were introduced. They are used for algebraic description of non-commutative fuzzy logic, see \cite{Haj}. For them the author \cite{151} generalized a well-known Mundici's representation theorem, see e.g. \cite[Cor 7.1.8]{CDM}, showing that every pseudo MV-algebra is always an interval in a unital $\ell$-group not necessarily Abelian.

From algebraic point of view of perfect MV-algebras, it was shown in \cite{DiLe} that every perfect MV-algebra $M$ can be represented as an interval in the lexicographic product, i.e. $M\cong \Gamma(\mathbb Z \lex G,(1,0))$. This result was extended also for perfect effect algebras \cite{177}.

This notion was generalized in \cite{Dv08} to $n$-perfect pseudo MV-algebras, they can be decomposed into $(n+1)$-comparable slices, and they can be represented in the form $\Gamma(\frac{1}{n}\mathbb Z \lex G,(1,0))$. $\mathbb R$-perfect pseudo MV-algebras can be represented in the form $\Gamma(\mathbb R\lex G,(1,0))$, see \cite{264}, if $G$ is Abelian, such MV-algebras were studied in \cite{DiLe2}. Recently, lexicographic MV-algebras were studied in \cite{DFL}, they have a representation in the form $\Gamma(H\lex G,(u,0))$, where $(H,u)$ is an Abelian linearly ordered group and $G$ is an Abelian $\ell$-group.

Thus we see that   MV-algebras and pseudo MV-algebras that can be represented in the form $\Gamma(H\lex G,(u,0))$ are intensively studied in the last period, see \cite{275} where $H$ was assumed to be Abelian. In this contribution, we continue in this study exhibiting the most general case of $(H,u)$ and $G$ when they are not assumed to be Abelian. We show that the crucial conditions are the existence of a lexicographic ideal, or equivalently, the possibility to decompose $M$ into comparable slices indexed by the elements of the interval $[0,u]_H$; we call such algebras $(H,u)$-perfect. In addition, we present also conditions when $M$ can be represented in the form $\Gamma(H\lex G,(u,b))$, where $b \in G^+$ is not necessarily the zero element.

The paper is organized as follows. The second section gathers the basic notions on pseudo MV-algebras. In the third section we introduce a lexicographic ideal, and we present a representation of a pseudo MV-algebra in the form $\Gamma(H\lex G,(u,0))$. Section 4 gives a categorical equivalence of the category of $(H,u)$-lexicographic pseudo MV-algebras to the category of $\ell$-groups. The final section will describe weakly $(H,u)$-perfect pseudo MV-algebras; they can be represented in the form $\Gamma(H\lex G,(u,b))$, where $b$ can be even strictly positive. Crucial notions for such algebras are a weakly lexicographic ideal as well as a weakly $(H,u)$-perfect pseudo MV-algebra.

\section{Basic Notions on Pseudo MV-algebras}

According to \cite{GeIo}, a {\it pseudo MV-algebra} or a {\it GMV-algebra} by \cite{Rac} is an algebra $(M;
\oplus,^-,^\sim,0,1)$ of type $(2,1,1,$ $0,0)$ such that the
following axioms hold for all $x,y,z \in M$ with an additional
binary operation $\odot$ defined via $$ y \odot x =(x^- \oplus y^-)
^\sim $$
\begin{enumerate}
\item[{\rm (A1)}]  $x \oplus (y \oplus z) = (x \oplus y) \oplus z;$

\item[{\rm (A2)}] $x\oplus 0 = 0 \oplus x = x;$

\item[{\rm (A3)}] $x \oplus 1 = 1 \oplus x = 1;$

\item[{\rm (A4)}] $1^\sim = 0;$ $1^- = 0;$

\item[{\rm (A5)}] $(x^- \oplus y^-)^\sim = (x^\sim \oplus y^\sim)^-;$

\item[{\rm (A6)}] $x \oplus (x^\sim \odot y) = y \oplus (y^\sim
\odot x) = (x \odot y^-) \oplus y = (y \odot x^-) \oplus
x;$\footnote{$\odot$ has a higher binding priority than $\oplus$.}

\item[{\rm (A7)}] $x \odot (x^- \oplus y) = (x \oplus y^\sim)
\odot y;$

\item[{\rm (A8)}] $(x^-)^\sim= x.$
\end{enumerate}

Any pseudo MV-algebra is a distributive lattice where (A6) and (A7) define the joint $x\vee y$ and the meet  $x\wedge y$ of $x,y$, respectively.

A pseudo MV-algebra $M$ is an {\it MV-algebra} if $x\oplus y=y\oplus x$ for all $x,y \in M$.

Orthodox examples of  pseudo MV-algebras are generated by unital $\ell$-groups not necessarily Abelian.

We note that a {\it po-group} (= partially ordered group) is a
group $(G;+,0)$ (written additively) endowed with a partial order $\le$ such that if $a\le b,$ $a,b \in G,$ then $x+a+y \le x+b+y$ for all $x,y \in G.$  We denote by $G^+=\{g \in G: g \ge 0\}$ the {\it positive cone} of $G.$ If, in addition, $G$
is a lattice under $\le$, we call it an $\ell$-group (= lattice
ordered group).  An element $u\in G^+$ is said to be a {\it strong unit}
(= order unit) if $G = \bigcup_n [-nu,nu]$, and the couple $(G,u)$ with a fixed strong unit $u$ is
said to be a {\it unital po-group} or a {\it unital $\ell$-group}, respectively. The {\it commutative center} of a group $H$ is the set $C(H)=\{h\in H: h+h' = h'+h, \ \forall h' \in H\}$. We denote by $[0,u]_H:=\{h \in H: 0\le h \le u\}$ for each $u \in H^+$.

Finally, two unital $\ell$-groups $(G,u)$ and $(H,v)$ are {\it isomorphic} if there is an $\ell$-group isomorphism $\phi:G \to H$ such that $\phi(u)=v$. In a similar way an isomorphism and a homomorphism of unital po-groups are defined.
For more information on po-groups and $\ell$-groups and for unexplained notions about them, see \cite{Dar, Fuc, Gla}.

By $\mathbb R$ and $\mathbb Z$ we denote the groups of reals and natural numbers, respectively.

Between pseudo MV-algebras and unital $\ell$-groups there is a very close connection:
If $u$ is a strong unit of a (not necessarily Abelian)
$\ell$-group $G$,
$$
\Gamma(G,u) := [0,u]
$$
and
\begin{eqnarray*}
x \oplus y &:=&
(x+y) \wedge u,\\
x^- &:=& u - x,\\
x^\sim &:=& -x +u,\\
x\odot y&:= &(x-u+y)\vee 0,
\end{eqnarray*}
then $(\Gamma(G,u);\oplus, ^-,^\sim,0,u)$ is a pseudo MV-algebra.

The basic representation theorem for pseudo MV-algebras is the following generalization \cite{151} of the Mundici famous result:

\begin{theorem}\label{th:2.1}
For any pseudo MV-algebra $(M;\oplus,^-,^\sim,0,1)$,
there exists a unique $($up to iso\-morphism$)$ unital $\ell$-group
$(G,u)$ such that $(M;\oplus,^-,^\sim,0,1)$ is isomorphic to $(\Gamma(G,u);\oplus,^-,^\sim,0,u)$. The
functor $\Gamma$ defines a categorical equivalence of the category
of pseudo MV-algebras with the category of unital $\ell$-groups.
\end{theorem}

We recall that in the category of pseudo MV-algebras objects are pseudo MV-algebras, and morphisms are homomorphisms of pseudo MV-algebras, whereas objects in the category of unital $\ell$-groups are unital $\ell$-groups $(G,u)$, and morphisms are homomorphisms of $\ell$-groups preserving fixed strong units.

We note that the class of pseudo MV-algebras is a variety whereas the class of unital $\ell$-groups is not a variety because it is not closed under infinite products.

Due to this result, if $M=\Gamma(G,u)$ for some unital $\ell$-group $(G,u)$, then $M$ is linearly ordered iff $G$ is a linearly ordered group, see \cite[Thm 5.3]{156}.

Besides a total operation $\oplus$, we can define a partial operation $+$ on any pseudo MV-algebra $M$  in such a way that $x + y$ is defined iff $x \odot y=0$ and then we set
$$
x+y := x\oplus y. \eqno(2.1)
$$
In other words, $x+y$ is precisely the group addition $x+y$ if the group sum $x+y$ is defined in $M$.

Let $A,B$ be two subsets of $M$. We define (i) $A\leqslant B$ if $a\le b$ for all $a\in A$ and all $b \in B$, (ii) $A\oplus B=\{a\oplus b: a\in A, b \in B\}$, and (iii) $A + B = \{a+b: $ if $a+b$ exists in $M$ for $a\in A,\ b \in B\}$, where the partial $+$ is defined by (2.1). We say that $A+B$ is {\it defined} in $M$ if $a+b$ exists in $M$  for each $a \in A$ and each $b \in B$. (iv) $A^-=\{a^-: a \in A\}$ and $A^\sim = \{a^\sim: a \in A\}.$

Using Theorem \ref{th:2.1}, we have  if $y\le x$, then $x\odot y^-=x-y$ and $y^\sim \odot x = -y+x$, where the subtraction $-$ is in fact the group subtraction in the representing unital $\ell$-group.

Given an element $x$ and any integer $n\ge 0$, we define
$$
0x := 0,\quad 1 x :=x, \quad (n+1)x:= (nx)+ x,
$$
if $nx$ and $(nx)+x$ are defined in $M$, where the simple $+$ is defined by (2.1). An element $x$ of $M$ is (i) {\it infinitesimal} if $nx$ is defined in $M$ for each integer $n \ge 1$, (ii) {\it co-infinitesimal} if $x^-$ is an infinitesimal. We denote by $\Infinit(M)$ the set of infinitesimal elements of $M$.

We recall that if $H$ and $G$ are two po-groups, then the {\it lexicographic product} $H \lex G$ is the group $H\times G$ which is endowed with the lexicographic order: $(h,g)\le (h_1,g_1)$ iff $h< h_1$ or $h=h_1$ and $g\le g_1$. The lexicographic product $H \lex G$ with non-trivial $G$ is an $\ell$-group iff $H$ is linearly ordered group and $G$ is an arbitrary $\ell$-group, \cite[(d) p. 26]{Fuc}. If $G=O$, the trivial group, then $H\lex O$ is an $\ell$-group that is isomorphic to $H$ for every $\ell$-group $H$ (not necessarily linearly ordered). If $u$ is a strong unit for $H$, then $(u,0)$ is a strong unit for $H\lex G$, and $\Gamma(H\lex G,(u,0))$ is a pseudo MV-algebra.

We say that a pseudo MV-algebra $M$ is {\it symmetric} if $x^-=x^\sim$ for all $x \in M$. A pseudo MV-algebra $\Gamma(G,u)$ is symmetric iff $u \in C(G)$, and the variety of symmetric pseudo MV-algebras is a proper subvariety of the variety of pseudo MV-algebras $\mathcal {PMV}$. For example, $\Gamma(\mathbb R\lex G,(1,0))$ is symmetric and it is an MV-algebra iff $G$ is Abelian.

An {\it ideal} of a pseudo MV-algebra $M$ is any non-empty subset $I$ of $M$ such that (i) $a\le b \in I$ implies $a \in I,$ and (ii) if $a,b \in I,$ then $a\oplus b \in I.$
An ideal $I$ is said to be (i) {\it maximal} if $I\ne M$ and it is not a proper subset of another ideal $J \ne M;$ we denote by $\mathcal M(M)$ the set of maximal ideals of $M$, (ii) {\it prime} if $x\wedge y \in I$ implies $x \in I$ or $y \in I$, and (iii) {\it normal} if $\{x\}\oplus I=I\oplus \{x\}$ for any $x \in M$.

If $I$ is a subset of $M$, then $\langle I\rangle$ denotes the least subalgebra of $M$ generated by $I$. If $I$ is an ideal of $M$ such that $I^-=I^\sim$, then it is easy to see that $\langle I \rangle = I\cup I^-=I\cup I^\sim$.

There is a one-to-one correspondence between normal ideals and congruences for pseudo MV-algebras, \cite[Thm 3.8]{GeIo}. The quotient pseudo MV-algebra over a normal ideal $I,$  $M/I,$ is defined as the set of all elements of the form $x/I := \{y \in M :
x\odot y^- \oplus y\odot x^- \in I\},$ or equivalently, $x/I := \{y \in M : x^\sim \odot y \oplus y^\sim \odot x \in I\}.$

The notion of a state is an analogue of a probability measure for pseudo MV-algebras. We say that a mapping $s$ from a pseudo MV-algebra $M$ into the real interval $[0,1]$ is a {\it state} if (i) $s(a+b)=s(a)+s(b)$ whenever $a+b$ is defined in $M$, and (ii) $s(1)=1$. We define the {\it kernel} of $s$ as the set  $\Ker(s)=\{a \in M: s(a)=0\}$. Then $\Ker(s)$ is a normal ideal of $M$.

Pseudo MV-algebras can be exhibited also in the realm of pseudo effect algebras with a special type of the Riesz Decomposition Property which are a non-commutative generalization of effect algebras introduced by \cite{FoBe}.

According to \cite{DvVe1, DvVe2}, a partial algebraic structure
$(E;+,0,1),$  where $+$ is a partial binary operation and 0 and 1 are constants, is called a {\it pseudo effect algebra} if, for all $a,b,c \in E,$ the following hold:
\begin{itemize}
\item[{\rm (PE1)}] $ a+ b$ and $(a+ b)+ c $ exist if and only if $b+ c$ and $a+( b+ c) $ exist, and in this case,
$(a+ b)+ c =a +( b+ c)$;

\item[{\rm (PE2)}] there are exactly one  $d\in E $ and exactly one $e\in E$ such
that $a+ d=e + a=1$;

\item[{\rm (PE3)}] if $ a+ b$ exists, there are elements $d, e\in E$ such that
$a+ b=d+ a=b+ e$;

\item[{\rm (PE4)}] if $ a+ 1$ or $ 1+ a$ exists,  then $a=0.$
\end{itemize}

If we define $a \le b$ if and only if there exists an element $c\in
E$ such that $a+c =b,$ then $\le$ is a partial ordering on $E$ such
that $0 \le a \le 1$ for any $a \in E.$ It is possible to show that
$a \le b$ if and only if $b = a+c = d+a$ for some $c,d \in E$. We
write $c = a \minusre b$ and $d = b \minusli a.$ Then

$$ (b \minusli a) + a = a + (a \minusre b) = b,
$$
and we write $a^- = 1 \minusli a$ and $a^\sim = a\minusre 1$ for any
$a \in E.$

If $(G,u)$ is a unital po-group, then $(\Gamma(G,u);+,0,u),$ where
the set $\Gamma(G,u):=\{g\in G: 0\le g \le u\}$ is endowed with the restriction of the group addition $+$ to $\Gamma(G,u)$ and with $0$ and $u$ as $0$ and $1$, is a pseudo effect algebra. Due to \cite{DvVe1, DvVe2}, if a pseudo effect algebra satisfies a special type of the Riesz Decomposition Property, RDP$_1$, then every pseudo effect algebra is an interval in some unique (up to isomorphism of unital po-groups) $(G,u)$ satisfying also RDP$_1$ such that $M \cong \Gamma(G,u)$.

We say that a mapping $f$ from one pseudo effect algebra $E$ onto a second one $F$ is a {\it homomorphism} (of pseudo effect algebras) if (i) $a,b\in E$ such that $a+b$ is defined in $E$, then $f(a)+f(b)$ is defined in $F$ and $f(a+b)=f(a)+f(b)$, and (ii) $f(1)=1$. Clearly, every homomorphism of pseudo effect algebras preserves $^-$ and $^\sim$. A bijective mapping $h:E\to F$ is an {\it isomorphism} if both $h$ and $h^{-1}$ are homomorphisms of pseudo effect algebras.

We say that a pseudo effect algebra $E$ satisfies RDP$_2$ property if $a_1+a_2=b_1+b_2$ implies that there are four elements $c_{11},c_{12},c_{21},c_{22}\in E$ such that (i) $a_1 = c_{11}+c_{12},$ $a_2= c_{21}+c_{22},$ $b_1= c_{11} + c_{21}$ and $b_2= c_{12}+c_{22}$, and (ii) $c_{12}\wedge c_{21}=0.$

In \cite[Thm 8.3, 8.4]{DvVe2}, it was proved that if $(M; \oplus,^-,^\sim,0,1)$ is a  pseudo MV-algebra, then $(M;+,0,1),$ where $+$ is defined by (2.1), is a pseudo effect algebra with RDP$_2.$  Conversely, if $(E; +,0,1)$ is a pseudo effect algebra with RDP$_2,$ then $E$ is a lattice, and by \cite[Thm 8.8]{DvVe2}, $(E; \oplus,^-,^\sim,0,1),$ where
$$
a\oplus b := (b^-\minusli (a\wedge b^-))^\sim,\quad a,b\in E, \eqno(2.2)
$$
is a pseudo MV-algebra.

\section{Lexicographic Ideals and Lexicographic Pseudo MV-algebras}

We say that a pseudo MV-algebra $M$ is {\it lexicographic} if there are a linearly ordered unital group $(H,u)$ and an $\ell$-group $G$ (both groups are not necessarily Abelian) such that $M \cong \Gamma(H\lex G,(u,0))$.  The main aim of this section is to show conditions when a pseudo MV-algebra is lexicographic. We show that such conditions are closely connected with the existence of a lexicographic ideal.

As a matter of interest, if $O$ is the zero group, then $\Gamma(O\lex G, (0,0))$ is a one-element pseudo MV-algebra. The pseudo MV-algebra $\Gamma(\mathbb Z\lex O,(1,0))$ is a two-element Boolean algebra. 

A normal ideal $I$ of a pseudo MV-algebra $M$ is said to be {\it retractive} if the canonical projection $\pi_I: M  \to M/I$ is retractive, i.e. there is a homomorphism $\delta_I: M/I \to M$ such that $\pi_I\circ \delta_I=id_{M/I}$. If a normal ideal $I$ is retractive, then $\delta_I$ is injective and $M/I$ is isomorphic to a subalgebra of $M$.

For example, if $M=\Gamma(H\lex G,(u,0))$ and $I=\{(0,g): g \in G^+\}$, then $I$ is a normal ideal, and due to  $M/I \cong \Gamma(H,u) \cong \Gamma(H\lex \{0\},(u,0)) \subseteq \Gamma(H\lex G,(u,0))$. If we set $\delta_I((h,g)/I):=(h,0)$, we see that $I$ is retractive.

We say that a normal ideal $I$ of a pseudo MV-algebra $M$ is {\it strict} if $x/I < y/I$ implies $x<y$.

\begin{definition}\label{de:lexid}
{\rm A normal ideal $I$ of a pseudo MV-algebra $M=\Gamma(G,u)$, $\{0\}\ne I \ne M$, is said to be {\it lexicographic} if
\begin{enumerate}
\item[{\rm (i)}] $I$ is strict;
\item[{\rm (ii)}] $I$ is retractive;
\item[{\rm (iii)}] $I$ is prime;
\item[{\rm (iv)}] for each $s,t \in [0,u]_H$, where $\Gamma(H,u):=M/I$, such that $s+t\le u$ and for each $x\in \pi_I^{-1}(\{s\})$ and $y \in \pi^{-1}_I(\{t\})$, we have $x+y-\delta_I(s+t)= (x-\delta_I(s))+(y-\delta_I(t))$, where $+$ and $-$ are counted in the group $G$,
\item[{\rm (v)}] for each $t \in [0,u]_H$ and each $x\in \pi^{-1}_I(\{t\})$, we have $x-\delta_I(t)=-\delta_I(t)+x$, where $+$ and $-$ are counted in the group $G$.
\end{enumerate}
}
\end{definition}

We note that if $M$ is an MV-algebra, then the concept of a lexicographic ideal coincides with a lexicographic ideal defined in \cite{DFL}, in addition, in such a case conditions (iv) and (v) are superfluous.

\begin{proposition}\label{pr:loc6}
Let $(H,u)$ be a linearly ordered unital group and let $G$ be an $\ell$-group. If we set $I=\{(0,g): g \in G^+\}$, then $I$ is a lexicographic ideal of $M=\Gamma(H\lex G,(u,0))$.
\end{proposition}

\begin{proof}
It is clear that $I$ is a normal ideal of $M$ as well as it is prime because $M/I \cong\Gamma(H,u)$ and the latter pseudo MV-algebra is linearly ordered.

We have $x/I =0/I$ iff $x \in I$.
Assume $(0,g)/I < (h,g')/I$. Then $(h,g) \notin I$ that yields $h>0$ and $(0,g)<(h,g')$. Hence, if $x/I <y/I$, then $(y-x)/I > 0/I$ and $y-x>0$ and $x<y$.

Since $M/I\cong\Gamma(H\lex \{0\},(u,0)) \subseteq \Gamma(H\lex G,(u,0))$, we see that $I$ is retractive. We have $I^-=I^\sim$, so that $\langle I \rangle = I \cup I^-$.


Let $h_1,h_2 \in [0,u]_H$ and $g_1,g_2\in G$ be such elements that $(h_1,g_1),(h_2,g_2)\in \Gamma(H\lex G,(u,0))$. Then $(h_i,g_i)\in \pi_I^{-1}(\{(h_i,0)/I\})$ and $\delta_i((h_i,0)/I)=(h_i,0)$, $i=1,2$. Hence, $(h_1+h_2,g_1+g_2)-(h_1+h_2,0)=(0,g_1)+(0,g_2)=((h_1,g_1)-(h_1,0))+ ((h_2,g_2)-(h_2,0))$ which proves that $I$ is a lexicographic ideal.

Finally, $(h,g)-(h,0)=(0,g)=-(h,0)+(h,g)$, so that (v) holds.
\end{proof}

Let $\mbox{LexId}(M)$ be the set of lexicographic ideals of $M$. Not every pseudo MV-algebra possesses a lexicographic ideal, e.g. the MV-algebra $M=\Gamma(\mathbb Z \lex \mathbb Z,(2,1))$ is such a case; it has a unique non-trivial ideal $I$, $I=\{(0,n): n \ge 0\}$, $M/I \cong \Gamma(\frac{1}{2}\mathbb Z,1)$ but $M$ does not contain any copy of $\Gamma(\frac{1}{2}\mathbb Z,1)$.  On the other side, it can happen that a pseudo MV-algebra could have more lexicographic ideals as it is in the following example.

\begin{example}\label{ex:3.3}
We define MV-algebras: $M_1=\Gamma(\mathbb Z \lex (\mathbb Z\lex \mathbb Z),(1,(0,0)))$, $M_2 = \Gamma((\mathbb Z \lex \mathbb Z)\lex \mathbb Z,((1,0),0))$, and $M=\Gamma(\mathbb Z \lex \mathbb Z\lex \mathbb Z,(1,0,0))$ which are  mutually isomorphic.
\end{example}

If we take the MV-algebra $M$ from Example \ref{ex:3.3}, we see that $I_1=\{(0,m,n): m > 0, n \in \mathbb Z \mbox{ or } m=0, n\ge 0\}$ and $I_2=\{(0,0,n): n \ge 0\}$ are only two lexicographic ideals of $M$ and $I_2 \subset I_1$.

Similarly as in \cite[Prop 7.1]{275}, we can show that if $I$ and $J$ are lexicographic ideals of $M$, then $I\subseteq J$ or $J\subseteq I$.

\begin{definition}\label{de:3.1}
{\rm Let $(H,u)$ be a linearly ordered group.
We say that a pseudo MV-algebra $M$ is $(H,u)$-{\it perfect}, if there is a system $(M_t: t \in [0,u]_H)$ of nonempty subsets of $M$ such that it is an $(H,u)$-{\it decomposition} of $M,$ i.e. $M_s \cap M_t= \emptyset$ for $s<t,$ $s,t \in [0,u]_ H$ and $\bigcup_{t\in [0,u]_ H}M_t = M$, and \begin{enumerate}

\item[{\rm (a)}]
$M_s \leqslant M_t$ for all $ s<t$, $s,t \in [0,u]_ H$;

\item[{\rm (b)}]
$M_t^- = M_{u-t}$ and $M_t^\sim=M_{-t+u}$ for each $t \in [0,u]_H$;

\item[{\rm (c)}]  if $x \in M_v$ and $y \in M_t$, then
$x\oplus y \in M_{v\oplus t},$ where $v\oplus t=\min\{v+t,u\}.$

\end{enumerate}}
\end{definition}

We note that in view of property (a), we have that if $x \in M_s$, $y \in M_t$, and $s<t$, then $x<y$. From (b), we have $M_0^-=M_u=M_0^\sim$. In addition, if $M$ is symmetric, then in (b) we have $M_t^- = M_{u-t} =M_t^\sim$ for each $t \in [0,u]_H$, and in \cite{275}, there was presented the notion of $(H,u)$-perfect pseudo MV-algebras only for symmetric pseudo MV-algebras.

For example, let us consider
$$
M =\Gamma(H\lex G, (u,0)). \eqno(3.1)
$$
We set $M_0=\{(0,g): g \in G^+\},$ $M_u:=\{(u,-g): g \in G^+\}$ and
for  $t \in [0,u]_ H\setminus \{0,u\},$ we define $M_t:=\{(t,g): g \in G\}.$ Then $(M_t: t \in [0,u]_H)$ is an $(H,u)$-decomposition of $M$ and $M$ is an $(H,u)$-perfect pseudo MV-algebra.

Sometimes we will write also $M=(M_t: t \in [0,u]_ H)$ for an $(H,u)$-perfect pseudo MV-algebra.

\begin{theorem}\label{th:3.2}
Let $M =(M_t: t \in [0,u]_ H)$ be an  $(H,u)$-perfect pseudo MV-algebra.

\begin{enumerate}

\item[{\rm (i)}] Let $a \in M_v,$  $b \in M_t$. If $v+t < u$, then
$ a+b $ is defined in $M$ and $a+b \in M_{v+t}$; if $a+b$ is defined
in $M$, then $v+t\le u$. If $a+b$ is defined in $M$ and $v+t=u$, then $a+b \in M_u$.

\item[{\rm (ii)}]
$M_v + M_t$ is defined in $M$ and $M_v + M_t = M_{v+t}$ whenever $v+t < u$.

\item[{\rm (iii)}] If $a \in M_v$ and $b \in M_t$, and
$v+t > u$, then $a+b$ is not defined in $M.$

\item[{\rm (iv)}] If $a \in M_v$ and $b \in M_t$, then $a\vee b \in M_{v\vee t}$ and $a\wedge b \in M_{v\wedge t}.$

\item[{\rm (v)}]
$M$ admits a state $s$ such that  $M_0 \subseteq \Ker(s)$.

\item[{\rm (vi)}]  $M_0$ is a normal ideal
of $M$ such that $M_0 + M_0 = M_0$ and $M_0\subseteq \Infinit(M).$

\item[{\rm (vii)}] The quotient pseudo MV-algebra $M/M_0\cong \Gamma(H,u).$

\item[{\rm (viii)}]
Let $M = (M'_t: t \in [0,u]_ H)$
be another $(H,u)$-decomposition of $M$ satisfying {\rm (a)--(c)} of Definition {\rm \ref{de:3.1}}, then
$M_t = M_t'$ for each $t \in [0,u]_ H.$

\item[{\rm (ix)}] $M_0$ is a prime ideal of $M$.

\end{enumerate}
\end{theorem}

\begin{proof}
It is similar to the proof of \cite[Thm 3.2]{275} where it was assumed that $M$ is symmetric, and therefore, we prove here only some items of them.

(ii) By (i), we have  $M_v +M_t \subseteq M_{v+t}$.
Suppose $z \in M_{v+t}$.  Then, for any $a\in M_v,$ we have $a \le z$.
Hence, $b =- a+z=a\rd z=a^\sim \odot z$ is defined in $M$ and $b \in M_w$ for
some $w \in [0,u]_ H.$ Since $z = a+b \in M_{v+t}\cap M_{v+w}$,  we conclude $t=w$ and $M_{v+t} \subseteq M_v + M_t.$

(iv) Inasmuch as $a\wedge b = (a\oplus b^\sim)-b^\sim$, we have by (c) of Definition \ref{de:3.1}, $(a\oplus b^\sim)-b^\sim \in M_s$, where $s =((v-t+u)\wedge u) -(-t+u)= v\wedge t.$ Using a de Morgan law, we have $a\vee b\in M_{v \vee t}.$
\end{proof}

We note that a $(\mathbb Z,1)$-perfect pseudo MV-algebra is in \cite{DDT} called {\it perfect}. Equivalently, a symmetric pseudo MV-algebra $M$ is perfect iff every element $x$ of $M$ is either infinitesimal or co-infinitesimal.

In addition, a $(\frac{1}{n}\mathbb Z,1)$-perfect pseudo MV-algebra is said to be $n$-perfect, see \cite{Dv08}.

Now we present the main result of this section, a representation of a pseudo MV-algebra with a lexicographic ideal as a lexicographic pseudo MV-algebra.

For each $\ell$-group $G$ and each unital linearly ordered group $(H,u)$, we define the pseudo MV-algebra
$$
\mathcal M_{H,u}(G):=\Gamma(H \lex G,(u,0)). \eqno(3.2)
$$

\begin{theorem}\label{th:local1}
Let $M$ be a pseudo MV-algebra and let $I$ be a lexicographic ideal of $M$. Then there are a linearly ordered unital group $(H,u)$ such that $E/I \cong \Gamma(H,u)$ and an $\ell$-group $G$ with $\langle I\rangle \cong \Gamma(\mathbb Z\lex G,(1,0))$ such that $M \cong \Gamma(H\lex G,(u,0))$.

In addition, if there is an $\ell$-group $G'$ such that $M \cong \Gamma(H\lex G',(u,0))$, then $G'$ is isomorphic to $G$.
\end{theorem}

\begin{proof}
According to a basic representation Theorem \ref{th:2.1}, we can assume that $M=\Gamma(K,v)$ for some unital $\ell$-group $(K,v)$. Since $I$ is lexicographic, then $I$ is normal and prime, so that $M/I$ is a linear pseudo MV-algebra. There is a linearly ordered unital group $(H,u)$ such that $M/I\cong \Gamma(H,u)$; without loss of generality, we assume $M/I=\Gamma(H,u)$.

Let $\pi_I:M\to M/I$ be the canonical projection. For any $t \in [0,u]_H$, we set $M_t:=\pi_I^{-1}(\{t\}$; then $M_0=I$. We assert that $(M_t: t \in [0,u]_H)$ is an $(H,u)$-decomposition of $M$. Indeed, since $\pi_I$ is surjective, every $M_t$ is non-empty. The decomposition $(M_t: t \in [0,u]_H)$ has the following properties: (a) Let $x \in M_s$ and $y \in M_t$ for $s<t,$ $s,t \in [0,u]_H$, then $x<y$. Indeed,  since $\pi_I(x)=s <t<\pi_I(y)$ and $x<y$ because $I$ is strict. (b) $M_t^-=M_{u-t}$ and $M_t^\sim=M_{-t+u}$
for each $t \in [0,u]_H$ which holds because $\pi_I$ is a homomorphism. (c) $M_s\oplus M_t\subseteq M_{s\oplus t}$ for $s,t \in [0,u]_H$, where $s\oplus t:=\min\{s+t,u\}$. Indeed, $x\in M_s$ and $y \in M_t$, then $\pi_I(x\oplus y)=\pi_I(x)\oplus \pi_I(y)=s\oplus t$, so that $M_s\oplus M_t \subseteq M_{s\oplus t}$.

By (vi) of Theorem \ref{th:3.2}, $M_0$ is an associative cancellative semigroup satisfying conditions of Birkhoff's Theorem  \cite[Thm XIV.2.1]{Bir}, \cite[Thm II.4]{Fuc},
which guarantee that $M_0$ is a positive cone of
a unique (up to isomorphism) directed po-group $G$. Since $M_0$ is a lattice, we have that $G$ is an $\ell$-group. In addition, since $\langle I \rangle = I \cup I^-$ is a perfect pseudo MV-algebra, we have by \cite[Prop 5.2]{DDT}, $\langle I\rangle \cong \Gamma(\mathbb Z\lex G,(1,0))$.

For each $t \in [0,u]_H$, we set $c_t=\delta_I(t)$. Then (i) $c_{s+t}=c_s+c_t$ if $s+t\le u$, (ii) $c_0=0$ and $c_u=1$, and (iii) $\{c_t\}=\delta_I(M/I)\cap M_t$, $t\in [0,u]_H$.

Take the $(H,u)$-perfect pseudo MV-algebra $\mathcal M_{H,u}(G)$ defined by (3.2), and define a mapping $\phi: M \to \mathcal M_{H,u}(G)$ by

$$
\phi(x):= (t, x - c_t)\eqno (3.3)
$$
whenever $x \in M_t$ for some $t \in [0,u]_ H,$ where $ x-c_t$ denotes the difference taken in the group $K$.

\vspace{2mm}
{\it Claim 1:} {\it  $\phi$ is a well-defined mapping.}
 \vspace{2mm}

Indeed, $M_0$  is in fact the positive cone of an $\ell$-group $G$ which is a subgroup of $K.$    Let $x \in M_t.$ For the element $x - c_t \in K,$ we define $(x-c_t)^+:= (x-c_t)\vee 0 = (x \vee c_t)-c_t \in M_0$ (when we use (iii) of Theorem \ref{th:3.2}) and similarly $(x -c_t)^- := -((x-c_t)\wedge 0) = c_t - (x\wedge c_t) \in M_0.$ This implies that $x-c_t= (x-c_t)^+ - (x-c_t)^-\in G.$

\vspace{2mm}
{\it Claim 2:} {\it The mapping $\phi$ is an injective and surjective homomorphism of pseudo effect algebras.}

\vspace{2mm}

We have $\phi(0)=(0,0)$ and $\phi(1)=(u,0).$ Let $x \in M_t.$ Using (v) of Definition \ref{de:lexid}, we have $x^- \in M_{u-t},$ $\phi(x)=(t,x-c_t)$ and $\phi(x^-) =(u-t, x^- -  c_{u-t}) = (u-t,-c_{u-t}+x^-)=(u-t,c_t-1+1-x)=(u-t,c_t-x)= (u,0)-(t,x - c_t)=\phi(x)^-.$ In an analogous way, $\phi(x^\sim)=\phi(x)^\sim.$

Now given $x,y \in M$ and let $x+y$ be defined in $M.$ Then $x\in M_{t_1}$ and $y \in M_{t_2}.$ Since $x\le y^-,$ we have $t_1 \le u-t_2$ so that $\phi(x) \le \phi(y^-)=\phi(y)^-$ which means  $\phi(x)+\phi(y)$ is defined in $\mathcal M_{H,u} (G).$ Using (iv) of Definition \ref{de:lexid}, we conclude $\phi(x+y) = (t_1+t_2, x+y - c_{t_1+t_2}) =
(t_1+t_2, x+y -(c_{t_1} + c_{t_2}))= (t_1,x-c_{t_1}) + (t_2,y- c_{t_2})=\phi(x)+\phi(y).$

Assume $\phi(x)\le \phi(y)$ for some $x\in M_{t}$ and $y \in M_v.$ Then $(t,x-c_t)\le (v, y - c_v).$ If $t=v,$ then $x-c_t\le y-c_t$ so that $x\le y.$  If $t<v,$ then $x \in M_t$ and $y\in M_v$ so that $x<y.$  Therefore, $\phi$ is injective.

To prove that $\phi$ is surjective, assume two cases: (i) Take $g \in G^+=M_0.$  Then $\phi(g)=(0,g).$ In addition $g^- \in M_u$ so that $\phi(g^-) = \phi(g)^-= (0,g)^- = (u,0)-(0,g)=(u,-g).$ (ii) Let $g \in G$ and $t$ with $0<t<u$ be given. Then $g = g_1-g_2,$ where $g_1,g_2 \in G^+=M_0.$ Since $c_t \in M_t,$ we have $g_2\le c_t$, so that $-g_2 + c_t=g_2\rd c_t$ exists in $M$ and it belongs to $M_t,$ which yields $g+c_t=g_1+ (- g_2+c_t) \in M_t.$  Hence,  $\phi(g+c_t)=(t,g)$ when we have used the property (iv) of Definition \ref{de:lexid}.

\vspace{2mm}
{\it Claim 3:}  {\it If $x\le y,$ then $\phi(y \minusli x)=\phi(y)\minusli \phi(x)$ and $\phi(x\minusre y)=\phi(x)\minusre \phi(y)$. In particular, $\phi(x^-)=\phi(x)^-$ and $\phi(x^\sim)=\phi(x)^\sim$ for each $x\in M$.  }
\vspace{2mm}

It follows from the fact that $\phi$ is a homomorphism of pseudo effect algebras.

\vspace{2mm}
{\it Claim 4:} $\phi(x\wedge y) = \phi(x)\wedge \phi(y)$ and $\phi(x\vee y)=\phi(x)\vee\phi(y).$
\vspace{2mm}

We have, $\phi(x), \phi(y) \ge \phi(x\wedge y).$ If $\phi(x), \phi(y) \ge \phi(w)$ for some $w \in M,$ we have $x,y \ge w$ and $x\wedge y \ge w.$ In the same way we deal with $\vee.$

\vspace{2mm}
{\it Claim 5:} {\it $\phi$ is a homomorphism of pseudo MV-algebras.}
\vspace{2mm}

It is necessary to show that $\phi(x\oplus y)=\phi(x)\oplus \phi(y).$
This follows straightforwardly from the previous claims and equality (2.2).

Consequently, $M$ is isomorphic to $\mathcal M_{H,u}(G)$ as pseudo MV-algebras.

If $M \cong \Gamma(H\lex G',(u,0))$ for some  $\ell$-group $G'$, let $\psi: \Gamma(H\lex G,(u,0)) \to \Gamma(H\lex G',(u,0))$ be an isomorphism of pseudo MV-algebras. Then $\psi$ induces another $(H,u)$-decomposition $(M_t': t \in [0,u]_H)$ of $\Gamma(H\lex G',(u,0))$, where $M_t'=\psi(M_t)$. By Theorem \ref{th:3.2}(viii), we see that $\psi(\{(0,g): g \in G^+\})= \{(0,g'): g' \in G'^+\}$ which proves that $G$ and $G'$ are isomorphic $\ell$-groups.
\end{proof}

We note that in Example \ref{ex:3.3}, the pseudo MV-algebra has two lexicographic ideals $I_1$ and $I_2$, so that it has two representations as lexicographic pseudo MV-algebras, namely one as $M_1$ and the second as $M_2$. One is $(\mathbb Z,1)$-perfect and the second is $(\mathbb Z\lex \mathbb Z,(1,0))$-perfect, and of course the linear unital groups $(H_1,u_1):=(\mathbb Z,1)$ and $(H_2,u_2):=(\mathbb Z\lex \mathbb Z,(1,0))$ are not isomorphic.

We say that a pseudo MV-algebra $M$ is $I$-{\it representable} if $I$ is a lexicographic ideal of $M$ and $M\cong \Gamma(H\lex G, (u,0))$, where $(H,u)$ is a linearly ordered unital group such that $M/I \cong \Gamma(H,u)$ and $G$ is an $\ell$-group such that $\langle I\rangle \cong \Gamma(\mathbb Z \lex G,(1,0))$;  the existences of $(H,u)$ and $G$ are guaranteed by Theorem \ref{th:local1}. Since the linearly ordered unital group $(H,u)$ is uniquely (up to isomorphism of unital $\ell$-groups) determined by the retractive ideal $I$, we can say also that $M$ is also $(H,u)$-{\it lexicographic}. This notion is well defined because if there is another lexicographic ideal $J$ of $M$ such that $M/J \cong \Gamma(H,u)$, then $(\pi_I^{-1}(\{t\}): t \in [0,u]_H)$ and $(\pi_J^{-1}(\{t\}): t \in [0,u]_H)$ are two $(H,u)$-decompositions of $M$, so that by Theorem \ref{th:3.2}(viii), they are the same, in particular $I=\pi_I^{-1}(\{0\})=\pi_J^{-1}(\{0\})=J$.

Now we define another notion that is very closely connected with lexicographic pseudo MV-algebras.

\begin{definition}\label{de:strong}
{\rm We say that an $(H,u)$-perfect pseudo MV-algebra $M=\Gamma(K,v)$ is {\it strongly} $(H,u)$-{\it perfect} if it is $(H,u)$-perfect, and if there is a system of elements $(c_t: t \in [0,u]_H)$ of $M$ such that
\begin{enumerate}
\item[{\rm (i)}] $c_t \in M_t$ for each $t \in [0,u]_H$;
\item[{\rm (ii)}] $c_{s+t}=c_s+c_t$ if $s+t \le u$;
\item[{\rm (ii)}] $c_u=1$;
\item[{\rm (iv)}] $(x+y)-c_{s+t}=(x-c_s)-(y-c_t)$ if $\in M_s$, $y\in M_t$, $s+t\le u$, where $+$ and $-$ are counted in the $\ell$-group $K$;
\item[{\rm (v)}] for each $t \in [0,u]_H$ and each $x \in M_t$, we have $x-c_t=-c_t+x$, where $+$ and $-$ are counted in the $\ell$-group $K$.
\end{enumerate}
}
\end{definition}

In view of (ii), we have $c_0+c_0=c_0$, so that $c_0=0$.

Now we show that, for any pseudo MV-algebra $M$, $I$-representability and strong $(H,u)$-perfectness are equivalent.

\begin{theorem}\label{th:3.8}
Let $M$ be a pseudo MV-algebra and $(H,u)$ be a linearly ordered group. The following assertions are equivalent:
\begin{itemize}
\item[{\rm (i)}] $M$ is $I$-representable and $M/I=\Gamma(H,u)$.
\item[{\rm (ii)}] $M$ is strongly $(H,u)$-perfect.
\item[{\rm (iii)}] $M$ is $(H,u)$-lexicographic.
\end{itemize}
\end{theorem}

\begin{proof}
(i) $\Rightarrow$ (ii). Let $I$ be a lexicographic ideal of $I$. From the proof of Theorem \ref{th:local1}, we see that if we put $M_t:= \pi_I^{-1}(\{t\})$ for each $t\in [0,u]$, then the system $(M_t: t \in [0,u]_H)$ is an $(H,u)$-decomposition of $M$. The system $(c_t:t\in [0,u]_H)$, where $c_t:=\delta_I(t)$ for $t \in [0,u]_H$, proves $M$ is strongly $(H,u)$-perfect.

(ii) $\Rightarrow$ (i). Let $(H,u)$ be a linearly ordered unital group and let $(M_t:t \in [0,u]_H)$ be an $(H,u)$-decomposition and let $(c_t: t \in [0,u]_H)$ be a system of elements of $M$ satisfying conditions (i)--(v) of Definition \ref{de:strong}. By Theorem \ref{th:3.2}(vi),(ix),(vii),  $I:=M_0$ is a normal and prime ideal of $M$ such that $M/I \cong \Gamma(H,u)$; without loss of generality, we assume $M/I=\Gamma(H,u)$. (Indeed, if $\iota_I:M/I \to \Gamma(H,u)$ is an isomorphism, there is a unique isomorphism $\delta'_I:\Gamma(H,u)\to M_I:=\delta_I'([0,u]_H)$ such that $\delta_I=\delta'_I \circ \iota_I$, and then we deal with $\pi_I':=\iota_I\circ \pi_I$, $\delta'_I$ and $id_{[0,u]_H}$ instead of $\pi_I$, $\delta_I$ and $id_{M/I}$, respectively.) In addition, the system $(M_t: t \in [0,u]_H)$ is a unique. Let $\pi_I:M \to M/I$ be the canonical homomorphism. Then $x\sim_I y$ iff there is $t \in [0,u]_H$ such that $x,y \in M_t$. Hence, $\pi_I^{-1}(\{t\})=M_t$ for each $t\in [0,u]_H$. In view of of (a) of Definition \ref{de:3.1}, we see that $I$ is a strict ideal of $M$.

Let $M':=\{c_t: t \in [0,u]_H\}$. In view of Definition \ref{de:strong}, we see that $M'$ is a pseudo effect algebra such that $c_t^-=c_{u-t}$ and $c_t^\sim=c_{-t+u}$ for each $t \in [0,u]_H$. This pseudo effect algebra is linearly ordered via $c_s < c_t$ iff $s< t$ so that $c_{s\wedge t}=c_s\wedge c_t$, where $\wedge$ is taken in the pseudo MV-algebra $M$. In addition, the order taken in $M$ and the one taken in the pseudo effect algebra $M'$ coincide in $M'$, so that $c_{t-s}=c_t - c_s$ and $c_{-s+t}=-c_s +c_t$ if $t>s$. In view of (2.2), we conclude, that $M'$ is a subalgebra of the pseudo MV-algebra $M$, and $M'=M_I$. If we define $\delta_I: \Gamma(H,u) \to M'$ via $\delta_I(t)=c_t$, we see that $\delta_I$ is a homomorphism such that $\pi_I \circ \delta_I=id_{M/I}$ which proves that $I$ is a retractive ideal. In view of (iv) of Definition \ref{de:strong}, we see that $I$ is a lexicographic ideal of $M$.

(i) $\Leftrightarrow$ (iii). It is evident.
\end{proof}

\begin{corollary}\label{co:3.9}
Let $M=(M_t: t \in [0,u]_H)$ be a strongly $(H,u)$-perfect pseudo MV-algebra with a fixed system of elements $(c_t: t \in [0,u]_H)$ satisfying conditions of Definition {\rm \ref{de:strong}}. Then $I:=M_0$ is a lexicographic ideal of $M$ such that $M/I \cong \Gamma(H,u)$, $M'=\{c_t: t \in [0,u]_H\}$ is a subalgebra of $M$ isomorphic to $\Gamma(H,u)$, and the mapping $\delta_I: M/I\to M'$ given by $\delta_I(t)=c_t$, $t \in [0,u]_H$, is an isomorphism such that $\pi_I \circ \delta_I =id_{M/I}$.
\end{corollary}

\begin{proof}
It follows from the proof of implication (ii) $\Rightarrow$ (i) of Theorem \ref{th:3.8}.
\end{proof}

As an additional corollary, we have that every strongly $(H,u)$-perfect pseudo MV-algebra is $(H,u)$-lexicographic, and consequently, it is lexicographic.

\begin{corollary}\label{co:3.10}
Let $(H,u)$ be a linearly ordered unital group. If $M$ is a strongly $(H,u)$-perfect pseudo MV-algebra, then there is a unique up to isomorphism $\ell$-group $G$ such that
$$
M\cong \Gamma(H \lex G,(u,0)).
$$
\end{corollary}

\begin{proof}
It follows from Theorem \ref{th:3.8} and Theorem \ref{th:local1}.
\end{proof}

\section{Categorical Equivalence}

As we have seen, $I$-representable pseudo MV-algebras, strongly $(H,u)$-perfect pseudo MV-algebras, and $(H,u)$-lexicographic pseudo MV-algebras, where $I$ is a lexicographic ideal of $M$ such that $M/I \cong \Gamma(H,u)$, are the same objects.

In this section, we establish the categorical equivalence of the category of $(H,u)$-lexicographic pseudo MV-algebras with the variety of $\ell$-groups. This extends the categorical representation perfect MV-algebras proved in \cite{DiLe}, the one of strongly $n$-perfect pseudo MV-algebras from \cite{Dv08}, the one of $\mathbb H$-perfect pseudo MV-algebras from \cite{264} as well as the categorical equivalence of lexicographic MV-algebras from \cite{DFL} with the variety of $\ell$-groups.

Thus we assume that $(H,u)$ is in this section a fixed linearly ordered unital group.

Let $M$ be an $(H,u)$-lexicographic pseudo MV-algebra, i.e. $M$ is a pseudo MV-algebra with a lexicographic ideal $I$ such that $M/I\cong \Gamma(H,u)$, and in addition, there is a subalgebra $M_I$ of $M$ which is isomorphic to $M/I$ and there is an isomorphism $\delta_I:M/I \to M_I$ such that $\pi_I \circ \delta_I =id_{M/I}$. Without loss of generalization, we will assume that $M/I =\Gamma(H,u)$; otherwise, we change $\pi_I$, $\delta_I$ and $id_{M/I}$ to $\pi_I':=\iota_I\circ \pi_I$,  $\delta_I'=\delta_I\circ \iota_I^{-1}$ (an isomorphism from $\Gamma(H,u)$ onto $M_I$), and $id_{\Gamma(H,u)}$, respectively, where $\iota_I:M/I \to \Gamma(H,u)$ is an isomorphism.

In other words, our $(H,u)$-lexicographic pseudo MV-algebra can be characterized by a quadruplet $(M,I,M_I,\delta_I)$, and $\delta_I$ will be now an isomorphism from $\Gamma(H,u)$ onto $M_I$.

For example, let $M= \Gamma(H \lex G,(u,0))$ for some $\ell$-group $G$. If we put $J=\{(0,g): g \in G^+\}$, then $\Gamma(H\lex G, (u,0))/J \cong \Gamma(H,u)$, $M_J:=\{(t,0): t\in [0,u]_H\}$ is a subalgebra of $\Gamma(H\lex G,(u,0))$ isomorphic to $\Gamma(H,u)\cong \Gamma(H\lex G, (u,0))/J$ with an isomorphism $\delta_J(t)=(t,0)$, $t \in [0,u]_H$, satisfying $\pi_J\circ \delta_J=id_{M/J}$.

Therefore, we define the category $\mathcal {LP}_s\mathcal{MV}_{H,u}$ of $(H,u)$-lexicographic pseudo MV-algebras whose objects are quadruplets $(M,I,M_I,\delta_I)$, where $I$ is a lexicographic ideal of $M$ such that $M/I \cong \Gamma(H,u)$, $M_I$ is a subalgebra of $M$ isomorphic to $M/I$ with an isomorphism $\delta_I: M/I \to M_I$ such that $\pi_I \circ \delta_I =id_{M/I}$. For example $(\Gamma(H \lex G,(u,0)), J, M_J, \delta_J)$ is an object of $\mathcal {LP}_s\mathcal{MV}_{H,u}$.

If $(M_1,I_1,M_{I_1},\delta_{I_1})$ and $(M_2,I_2,M_{I_2},\delta_{I_2})$ are two objects of $\mathcal {LP}_s\mathcal{MV}_{H,u}$, then a morphism $f: (M_1,I_1,M_{I_1},\delta_{I_1}) \to (M_2,I_2,M_{I_2},\delta_{I_2})$ is a homomorphism of pseudo MV-algebras $f: M_1 \to M_2$ such that

$$ f(I_1) \subseteq I_2,\quad f(M_{I_1})\subseteq M_{I_2}, \ \mbox{ and }
f\circ \delta_{I_1}=\delta_{I_2}.
$$

It is straightforward to verify that $\mathcal {LP}_s\mathcal{MV}_{H,u}$ is a well-defined category.

Now let $\mathcal{LG}$ be the category whose objects are $\ell$-groups  and morphisms are homomorphisms of $\ell$-groups.

Define a mapping $\mathcal M^s_{H,u}: \mathcal{LG} \to  \mathcal {LP}_s\mathcal{MV}_{H,u}$ as follows: for $G\in \mathcal{LG},$ let
$$
\mathcal M^s_{H,u}(G):= (\Gamma(H\lex G,(u,0)),J,H_J,\delta_J)\eqno(4.1)
$$
and if $h: G \to G_1$ is an $\ell$-group homomorphism, then

$$
\mathcal M^s_{H,u}(h)(t,g)= (t, h(g)), \quad (t,g) \in \Gamma(H\lex G,(u,0)). \eqno(4.2)
$$

\begin{proposition}\label{pr:4.1}
$\mathcal M^s_{H,u}$ is a well-defined functor that is a faithful and full
functor from the category $\mathcal{LG}$ of $\ell$-groups  into the
category $\mathcal{LP}_s\mathcal{MV}_{H,u}$ of $(H,u)$-lexicographic pseudo MV-algebras.
\end{proposition}

\begin{proof}
First we show that $\mathcal M^s_{H,u}$ is a well-defined functor. Alias, we have to show that if $h$ is a morphism of $\ell$-groups, then $\mathcal M^s_{H,u}(h)$ is a morphism in the category $\mathcal{LP}_s\mathcal{MV}_{H,u}$. Let $\mathcal M^s_{H,u}(G):= (\Gamma(H\lex G,(u,0)),J,H_J,\delta_J)$ and $\mathcal M^s_{H,u}(G_1):= (\Gamma(H\lex G_1,(u,0)),J_1,H_{J_1},\delta_{J_1})$. Then $\mathcal M^s_{H,u}(h)(J)\subseteq J_1$ and $\mathcal M^s_{H,u}(h)(H_J)\subseteq H_{J_1}$. Check, let $t \in \Gamma(H,u)$, then $\mathcal M^s_{H,u}(h)\circ \delta_J(t)=\mathcal M^s_{H,u}(h)(t,0)=(t,0)=\delta_{J_1}(t)$.

Let $h_1$ and $h_2$ be two morphisms from $G$
into $G'$ such that $\mathcal M^s_{H,u}(h_1) = \mathcal M^s_{H,u}(h_2)$. Then
$(0,h_1(g)) = (0,h_2(g))$ for each $g \in G^+$, consequently $h_1 =
h_2.$

To prove that $\mathcal M^s_{H,u}$ is a full  functor, let
$f: (\Gamma(H\lex G, (u,0)),J,H_J,\delta_J) \to (\Gamma(H\lex G', (u,0)),J',H_{J'},\delta_{J'})$ be a morphism from $\mathcal{LP}_s\mathcal{MV}_{H,u}$. We claim that $f(t,0)=(t,0)$ for each $t \in \Gamma(H,u)$. Indeed, we have $f(t,0)=f(\delta_J(t))=\delta_{J'}(t)=(t,0)$.

In addition,  we have $f(0,g)
= (0,g')$ for a unique $g' \in G'^+$. Define a mapping $h:\ G^+ \to
G'^+$ by $h(g) = g'$ iff $f(0,g) =(0,g').$ Then $h(g_1+g_2) = h(g_1)
+ h(g_2)$ if $g_1,g_2 \in G^+.$
Assume now that $g \in G$ is arbitrary. Then $g=g^+-g^-,$ where $g^+=g\vee 0$ and $g^-=-(g\wedge 0)$, and $g=-g^-+g^+$. If $g=g_1-g_2,$ where $g_1,g_2 \in G^+$, then $g^++g_2=g^-+g_1$ and $h(g^+)+ h(g_2)=h(g^-)+h(g_1)$ which shows that $h(g) = h(g_1) - h(g_2)$ is a
well-defined extension of $h$ from $G^+$ onto $G$.

If $0\le g_1 \le g_2$, then $(0,g_1)\le (0,g_2),$
which means that $h$ is a mapping preserving the partial order.

We have yet to show that $h$ preserves $\wedge$ in $G$, i.e., $h(a \wedge b) = h(a) \wedge h(b)$ whenever $a,b \in G.$ Let $a=  a^+- a^-$ and $b=
b^+- b^-$, and $a =-a^- +a^+$, $b = -b^- + b^+$. Since , $h((a^+
+b^-) \wedge (a^- + b^+)) = h(a^+ +b^-) \wedge h(a^- + b^+).$
Subtracting $h(b^-)$ from the right hand and $h(a^-)$ from the left
hand, we obtain the statement  in question.

By Theorem \ref{th:2.1}, the homomorphism $f: \Gamma(H\lex G,(u,0))\to \Gamma(H\lex G',(u,0))$ can be uniquely extended to a morphism of unital $\ell$-groups $\bar f: (H\lex G,(u,0))\to (H\lex G',(u,0))$. Then $f(t,g)=\bar f(t,g)=\bar f(t,0)+\bar f(0,g)= (t,0)+(0,h(g))=(t,h(g))$.

Finally, we have proved that $h$ is a homomorphism of $\ell$-groups, and $\mathcal M^s_{H,u}(h) = f$ as claimed.
\end{proof}

We note that by a {\it universal group}  for a
pseudo MV-algebra $M$ we mean a pair $(G,\gamma)$ consisting of an
$\ell$-group $G$ and a $G$-valued measure $\gamma :\, M\to G^+$
(i.e., $\gamma (a+b) = \gamma(a) + \gamma(b)$ whenever $a+b$ is
defined in $M$) such that the following conditions hold: {\rm (i)}
$\gamma(M)$ generates ${ G}$. {\rm (ii)} If $K$ is a group and
$\phi:\, M\to K$ is an $K$-valued measure, then there is a group
homomorphism ${\phi}^*:{ G}\to K$ such that $\phi ={\phi}^*\circ
\gamma$.

Due to \cite{151}, every pseudo MV-algebra admits a universal group,
which is unique up to isomorphism, and $\phi^*$ is unique. The
universal group for $M = \Gamma(G,u)$ is $(G,id)$ where $id$ is an
embedding of $M$ into $G$.

\begin{proposition}\label{pr:4.2}
The functor $\mathcal M^s_{H,u}$ from the
category $\mathcal{LG}$ into  $\mathcal{LP}_s\mathcal{MV}_{H,u}$ has  a left-adjoint.
\end{proposition}

\begin{proof}
The proof follows the ideas of the proof of \cite[Prop 8.3]{275}, but we present it in its fullness to be self-contained.

Every object $(M,I,M_I,\delta_I)$ in $\mathcal{LP}_s\mathcal{MV}_{H,u}$, has a universal arrow $(G,f)$, i.e., $G$ is an object in $\mathcal{LG}$ and $f$ is a homomorphism from the pseudo MV-algebra
$M$ into $\mathcal M^s_{H,u}(G)$ such that if $G'$ is an object from $\mathcal{LG}$ and $f'$ is a homomorphism from $M$ into ${\mathcal  M}^s_{H,u}(G')$, then
there exists a unique morphism $f^*:\, G \to G'$ such that ${\mathcal
M}^s_{H,u}(f^*)\circ f = f'$.

Since by Theorem \ref{th:3.8}, $M$ is also a strongly $(H,u)$-perfect pseudo MV-algebra with an  $(H,u)$-decomposition $(M_t: t \in [0,u]_H)$ and with a family $(c_t: t \in [0,u]_H)$  of elements of $M$ satisfying Definition \ref{de:strong}, by Theorem \ref{th:local1}, there is a unique (up to isomorphism of $\ell$-groups) $\ell$-group $G$  such that $M \cong \Gamma(H \lex G,(u,0)).$ By \cite[Thm 5.3]{151}, $(H \lex G, \gamma)$ is a universal group for $M,$ where $\gamma: M \to  \Gamma(H \lex G, (u,0))$ is defined by $\gamma(a) = (t,a -c_t)$ if $a \in M_t$, see (3.3).

We assert that $(G,\gamma)$ is a universal arrow for $(M,I,M_I,\delta_I)$. Clearly we have that $\gamma:(M,I,M_I,\delta_I) \to (\Gamma(H\lex G,(u,0)),J,H_{J},\delta_{J})$ is a morphism. In addition, assume that $f':(M,I,M_I,\delta_I) \to (\Gamma(H\lex G',(u,0)),J',H_{J'},\delta_{J'})$ is an arbitrary morphism. There is a unique homomorphism of unital $\ell$-groups $\alpha : (H\lex G, (u,0))\to (H\lex G',(u,0))$ such that $\Gamma(\alpha) \circ \gamma = f'$, where $\Gamma$ is a functor  from the category of unital $\ell$-groups into the category of pseudo MV-algebras, see Theorem \ref{th:2.1}. Since $f'$ and $\gamma$ are morphisms of the category $\mathcal{LP}_s\mathcal{MV}_{H,u}$, then  $\gamma\circ \psi = \psi_J$ and $f'\circ \delta_I = \delta_{J'}$ which entail $\Gamma(\alpha)(J) = \Gamma(\alpha) \circ \gamma \circ \gamma^{-1}(J) = f'\circ(\gamma^{-1}(J))= f'(I)\subseteq J'$, $\Gamma(\alpha)(H_j)=\Gamma(\alpha)\circ \gamma \circ \gamma^{-1}(H_j)= f'(M_I)\subseteq J'$, and $\Gamma(\alpha) \circ \delta_J(t) = \Gamma(\alpha)\circ \gamma \circ \delta_I(t)=f'\circ\delta_I(t)=\psi_{J'}(t)$, $t \in \Gamma(H,u)$. This proves that $\Gamma(\alpha)$ is a morphism from $(\Gamma(H\lex G,(u,0)),J,H_J,\delta_J)$ into $(\Gamma(H\lex G',(u,0)),J',H_{J'},\delta_{J'})$.

Then for each $g\in G$, there is a unique $g'\in G'^+$ such that $\Gamma(\alpha)(0,g)=(0,g')$. Hence, there is an $\ell$-group homomorphism $\beta: G\to G'$ such that $\Gamma(\alpha)(0,g)=(0,\beta(g))$ for each $g \in G$. This gives $\mathcal M^s_{H,u}(\beta)(h,g)= (h,\beta(g))=(h,0)+(0,\beta(g))=\alpha(h,0)+\alpha(0,g)= \alpha(h,g)$ and  $\mathcal M^s_{H,u}(\beta) \circ \gamma = f'$, consequently, $(G,\gamma)$ is a universal arrow for $(M,I,M_I,\delta_I)$.
\end{proof}

Define a mapping ${\mathcal  P}^s_{H,u}: \mathcal  {LP}_s\mathcal{MV}_{H,u}\to \mathcal{LG}$
via ${\mathcal  P}^s_{H,u}(M,I,M_I,\delta_I) := G$ whenever $(H\lex  G, f)$ is a
universal group for $M$. It is clear that if $f_0$ is a morphism
from $(M,I,M_I,\delta_I)\in \mathcal  {LP}_s\mathcal{MV}_{H,u}$ into another one $(N,I_N,N_{I_N},\delta_N)$, then $f_0$ can be uniquely extended to an $\ell$-group homomorphism ${\mathcal  P}^s_{H,u} (f_0)$ from $G$ into $G_1$, where $(H\lex G_1, f_1)$ is a universal group for an $(H,u)$-lexicographic
pseudo MV-algebra $N$. Therefore, we have the following statement.

\begin{proposition}\label{pr:4.3}
The mapping ${\mathcal  P}^s_{H,u}$ is a functor from the
category $\mathcal {LP}_s\mathcal{MV}_{H,u}$ into the category $\mathcal{LG}$ which is a
left-adjoint of the functor ${\mathcal  M}^s_{H,u}$.
\end{proposition}

Now we present  the basic result of this section on a categorical equivalence of the category of $(H,u)$-lexicographic pseudo MV-algebras and the category of $\mathcal{LG}.$

\begin{theorem}\label{th:4.4}
The functor ${\mathcal  M}^s_{H,u}$ defines a categorical
equivalence of the category $\mathcal{LG}$   and the
category $\mathcal {LP}_s\mathcal{MV}_{H,u}$ of $(H,u)$-lexicographic pseudo MV-algebras.

In addition, suppose that $h:\ {\mathcal  M}^s_{H,u}\mathbb (G) \to {\mathcal  M}^s_{H,u}(G')$ is a
homomorphism of pseudo  MV-algebras, then there is a unique homomorphism
$f:\ G \to G'$ of  $\ell$-groups such that $h = {\mathcal  M}^s_{H,u}(f)$, and
\begin{enumerate}
\item[{\rm (i)}] if $h$ is surjective, so is $f$;
 \item[{\rm (ii)}] if $h$ is  injective, so is $f$.
\end{enumerate}
\end{theorem}

\begin{proof}
In view of \cite[Thm IV.4.1]{MaL}, it is
necessary to show that, for any $(H,u)$-lexicographic pseudo MV-algebra $(M,I,M_I,\delta_I)$, there is an object $G$ in $\mathcal{LG}$ such that ${\mathcal  M}^s_{H,u}(G)$ is isomorphic to $(M,I,M_I,\delta_I)$. To establish that, we take a universal arrow $(H\lex G, f)$ of $M$. Then ${\mathcal  M}^s_{H,u}(G)$ and $(M,I,M_I,\delta_I)$ are isomorphic.
\end{proof}

\section{Weakly Lexicographic Pseudo MV-algebras}

In this section, we generalize the notion of a retractive ideal, a lexicographic ideal, and a lexicographic pseudo MV-algebra in order to characterize pseudo MV-algebras that can be represented in the form $\Gamma(H \lex G,(u,b))$, where $b\in G^+$ is not necessarily the zero element.

First we introduce another notion of $(H,u)$-perfect pseudo MV-algebra.

\begin{definition}\label{de:weak}
{\rm We say that an $(H,u)$-perfect pseudo MV-algebra $M=\Gamma(K,v)$ is {\it weakly} $(H,u)$-{\it perfect} if it is $(H,u)$-perfect, and if there is a system of elements $(c_t: t \in [0,u]_H)$ of $M$ such that
\begin{enumerate}
\item[{\rm (i)}] $c_t \in M_t$ for each $t \in [0,u]_H$;
\item[{\rm (ii)}] $c_{s+t}=c_s+c_t$ if $s+t \le u$;
\item[{\rm (iii)}] $(x+y)-c_{s+t}=(x-c_s)-(y-c_t)$ if $\in M_s$, $y\in M_t$, $s+t\le u$, where $+$ and $-$ are counted in the $\ell$-group $K$;
\item[{\rm (iv)}] for each $t \in [0,u]_H$ and each $x \in M_t$, we have $x-c_t=-c_t+x$, where $+$ and $-$ are counted in the $\ell$-group $K$.
\end{enumerate}
}
\end{definition}

For example, let $M=\Gamma(H \lex G,(u,b))$ for some $b\in G^+$. We set $c_t=(t,0)$, $t \in [0,u]_H$, then $M$ is weakly $(H,u)$-perfect.

In view of (ii), we have $c_0+c_0=c_0$, so that $c_0=0$. Comparing with Definition \ref{de:strong}, we see that we do not assume that $c_u=1$, therefore, the subset $\{c_t: t\in [0,u]_H\}$ is not necessarily a subalgebra of $M$. Of course every strongly $(H,u)$-perfect pseudo MV-algebra is a weakly $(H,u)$-perfect one, but as we show below, the converse is not true in general.

We note that according to Theorem \ref{th:local1} and Theorem \ref{th:3.8}, every strongly $(H,u)$-perfect pseudo MV-algebra $M$ is of the form $M \cong \Gamma(H\lex G, (u,0))$ and it admits a lexicographic ideal $I$ such that $M/I \cong \Gamma(H,u)$. The MV-algebra $M=\Gamma(\mathbb Z \lex \mathbb Z,(2,1))$ is weakly $(\mathbb Z,2)$-perfect; it contains elements $c_0=(0,0)$, $c_1=(1,0)$, $c_2=(2,0)$, however, as it was already mentioned, it has a unique non-trivial ideal $I=\{(0,n): n\ge 0\}$ and it is not lexicographic because $M$ does not contain any copy of $\Gamma(\frac{1}{2}\mathbb Z,1)$.

On the other hand, it can happen, that a weakly $(H,u)$-perfect pseudo MV-algebra is also strongly $(H,u)$-perfect. Indeed, MV-algebras $M_1=\Gamma(\mathbb Z \lex \mathbb Z,(2,2))$ and $M_2 =\Gamma(\mathbb Z \lex \mathbb Z,(2,-2))$ are isomorphic to $M= \Gamma(\mathbb Z \lex \mathbb Z,(2,0))$; we define isomorphisms $\theta_i: M\to M_i$, $i=1,2$, such that $\theta_1(0,n)= (0,n)$,  $\theta_1(1,n)= (1,n+1)$, $\theta_1(2,n)=(2,n+2)$ and $\theta_2(0,n)= (0,n)$,  $\theta_2(1,n)= (1,n-1)$, $\theta_2(2,n)=(2,n-2)$. In $M_1$ we can define two families $c_1=(0,0),$ $c_2=(1,0)$, $c_3=(2,0)$ which shows that $M_1$ is weakly $(\mathbb Z,2)$-perfect, and the family $c_1'=(0,0),$ $c_2'=(1,1)$, $c_3'=(2,2)$ which shows that it is also strongly $(\mathbb Z,2)$-perfect. The first family does not form a subalgebra of $M$, and the second one does form. For the MV-algebra $M_2$ we have two analogous families: $c_1=(0,0),$ $c_2=(1,-2)$, $c_3=(2,-4)$ and  $c_1'=(0,0),$ $c_2'=(1,-1)$, $c_3'=(2,-2)$.

Now we introduce a weaker form of a retractive ideal.

A normal ideal $I$ of a pseudo MV-algebra $M$ is said to be {\it weakly retractive} if the canonical projection $\pi_I: M  \to M/I$ is weakly retractive, i.e. there is a mapping $\delta_I: M/I \to M$ such that (i) $\pi_I\circ \delta_I= id_{M/I}$, (ii) $\delta_I(x/I+ y/I)=\delta_I(x/I)+ \delta_I(y/I)$ whenever $x/I +y/I \le 1/I$, where $+$ is the partial addition induced by $\oplus$ in pseudo MV-algebras.
Then (1) $\delta_I(0/I)=0$, (2) $\delta_I(x/I)<1$ whenever $x/I < 1/I$, (3) $\delta_I$ is injective.

We note that a weakly retractive ideal is retractive whenever $\delta_I(1/I)=1$.

In addition, we define a weakly lexicographic ideal:

\begin{definition}\label{de:wlexid}
{\rm A normal ideal $I$ of a pseudo MV-algebra $M=\Gamma(G,u)$, $\{0\}\ne I \ne M$, is said to {\it weakly lexicographic} if
\begin{enumerate}
\item[{\rm (i)}] $I$ is strict;
\item[{\rm (ii)}] $I$ is weakly retractive;
\item[{\rm (iii)}] $I$ is prime;
\item[{\rm (iv)}] for each $s,t \in [0,u]_H$, where $\Gamma(H,u):=M/I$, such that $s+t\le u$ and for each $x\in \pi_I^{-1}(\{s\})$ and $y \in \pi^{-1}_I(\{t\})$, we have $x+y-\delta_I(s+t)= (x-\delta_I(s))+(y-\delta_I(t))$, where $+$ and $-$ are counted in the group $G$,
\item[{\rm (v)}] for each $t \in [0,u]_H$ and each $x\in \pi^{-1}_I(\{t\})$, we have $x-\delta_I(t)=-\delta_I(t)+x$, where $+$ and $-$ are counted in the group $G$.
\end{enumerate}
}
\end{definition}

For example, $M=\Gamma(\mathbb Z\lex \mathbb Z,(2,1))$ has no lexicographic ideal, but $I=\{(0,n): n \ge 0\}$ is a weakly lexicographic ideal of $M$.

If $M=\Gamma(H\lex G,(u,b))$ with $b \in G^+$, then $I=\{(0,g): g \in G^+\}$ is a weakly lexicographic ideal of $M$, and $(M_t: t \in [0,u]_H)$, where $M_t=\{(t,g): (t,g)\in M\}$, is an $(H,u)$-decomposition of $M$. In particular, $I^-=I^\sim$.

Now we characterize $(H,u)$-perfect pseudo MV-algebras that can be represented in the form $\Gamma(H\lex G, (u,b))$ for some $b\in G^+$.

\begin{theorem}\label{th:weak}
Let $M$ be a pseudo MV-algebra and let $I$ be a weakly lexicographic ideal of $M$. Then there are a linearly ordered unital group $(H,u)$ such that $E/I \cong \Gamma(H,u)$, an $\ell$-group $G$ with $\langle I\rangle \cong \Gamma(\mathbb Z\lex G,(1,0))$ and an element $b\in G^+$ such that $M \cong \Gamma(H\lex G,(u,b))$.

In addition, if there is an $\ell$-group $G'$ such that $M \cong \Gamma(H\lex G',(u,b'))$ where $b' \in G'^+$, then $G'$ is isomorphic to $G$.
\end{theorem}

\begin{proof}
We follow the main steps of the proof of Theorem \ref{th:local1}. Thus let $M =\Gamma(K,v)$ for some unital $\ell$-group $(K,v)$, $I$ be a weakly lexicographic ideal of $M$, and let $\pi_I:M \to M/I$ be the canonical projection. Since $I$ is prime, there is a linearly ordered unital group $(H,u)$ such that $M/I \cong \Gamma(H,u)$; without loss of generality, we assume $M/I=\Gamma(H,u)$. Then $(M_t: t \in [0,u]_H)$ is an $(H,u)$-decomposition of $M$, where $M_t:=\pi_I^{-1}(\{t\})$, $t \in [0,u]_H$.

Being $M_0=I$, $M_0$ is an associative cancellative semigroup, $M_0$ is a positive cone of an $\ell$-group $G$. By \cite[Prop 5.2]{DDT}, $\langle I\rangle \cong \Gamma(\mathbb Z\lex G,(1,0))$.

Given $t \in [0,u]_H$ we put $c_t=\delta_I(t)$. Since $I$ is weakly lexicographic, $M$ with $(c_t: t \in [0,u]_H)$ is weakly $(H,u)$-perfect, $\{c_t\} = M_t \cap \delta_I(M/I)$. Then $\delta_I(u)\le 1$ and we put $b=1- \delta_I(u)\in G^+$, where $-$ is subtraction counted in the $\ell$-group $G \subseteq K$. We note that according to (v) of Definition \ref{de:wlexid}, $1-c_u=-c_u+1$.

\vspace{2mm}
{\it Claim:} {\it $c_t+b=c_t+b$ for every $t\in [0,u]_H$.}
\vspace{2mm}

Indeed, according to (iv) of Definition \ref{de:wlexid}, we have $b=1-c_u=
x+x^\sim -c_{t+(-t+u)}= x-c_t + x^\sim - c_{-t+u}= -c_t+x-x +1-(-c_t+c_u)=
-c_t+1-c_u+c_t=-c_t+b+c_t=b$, so that $b+c_t=c_t+b$.

We define a pseudo MV-algebra
$$
\mathcal M_{H,u}(G,b):=\Gamma(H\lex G,(u,b)),
$$
and define a mapping $\phi: M\to  \mathcal M_{H,u}(G,b)$ by
$$
\phi(x):=(t,x-c_t) \eqno(5.1)
$$
whenever $x \in M_t$ for some $t \in [0,u]_H$, where $x-c_t$ is a difference taken in the group $K$.

Similarly as in the proof of Theorem \ref{th:local1}, $\phi$ is a well-defined mapping such that (i) $\phi(0)=(0,0)$, $\phi(1)=(u,1-c_u)=(1,b)$. (ii) If $x\in M_t$, then $x^\sim \in M_{-t+u}$, $x^-\in M_{u-t}$, so that $\phi(x)=(t,x-c_t)$ and $\phi(x^\sim)=(-t+u,(-x+1)-(c_{-t+u}))=(-t+u,-x+b+c_t)$. On the other side, $\phi(x)^\sim = -(t,x-c_t) +(u,b)=(-t+u,c_t-x+b)=\phi(x^\sim)$ when we have used (v) of Definition \ref{de:wlexid} and Claim. In a similar way we have, $\phi(x^-)= (u-t,x^- - c_{u-t})=(u-t, -c_{u-t}+ 1-x)= (u-t, c_t-c_u+ 1-x) =(u-t, b+c_t -x)=(u,b)-(t,x-c_t)= \phi(x)^-$.

Using the ideas of the proof of Theorem \ref{th:local1} and (iv) of Definition \ref{de:wlexid}, we have $\phi(x+y)=\phi(x)+\phi(y)$, $\phi$ is injective and surjective which proves that $\phi$ is an isomorphism of pseudo effect algebras, consequently, using (2.2), $\phi$ is an isomorphism of the pseudo MV-algebras $M$ and $\mathcal M_{H,u}(G,b)$.

Finally, if $M \cong \Gamma(H\lex G',(u,b'))$ for some $G'$ and $b' \in G'^+$, then in a similar manner as at the end of the proof of Theorem \ref{th:local1}, we can prove that $G$ and $G'$ are isomorphic $\ell$-groups.
\end{proof}

It is worthy to recall that in Theorem \ref{th:weak}, if $\Gamma(H\lex G,(u,b))$ is isomorphic to $\Gamma(H\lex G',(u,b'))$, where $b \in G^+$ and $b'\in G'^+$, then $G$ and $G'$ are isomorphic $\ell$-groups, but $b$ does not map necessarily to $b'$. Indeed, take $\Gamma(\mathbb Z \lex \mathbb Z, (2,0))$ and $\Gamma(\mathbb Z \lex \mathbb Z,(2,2))$. It was already proved that they are isomorphic, but $b=0$ does not map to $b=2$ under any isomorphism from $\mathbb Z$ onto itself.

Using ideas from the proof of Theorem \ref{th:3.8}, it is possible to show that a pseudo MV-algebra is weakly $(H,u)$-perfect iff $M$ has a weakly lexicographic ideal $I$ such that $M/I \cong \Gamma(H,u)$. Hence, such pseudo MV-algebras can be called also {\it weakly} $(H,u)$-{\it lexicographic} pseudo MV-algebras.

Finally we present a categorical equivalence of the category of weakly $(H,u)$-lexicographic pseudo MV-algebras with the category of pointed $\ell$-groups in an analogous way as it was done in the previous section.

Let $M$ be a weakly $(H,u)$-lexicographic pseudo MV-algebra. It has a weakly lexicographic ideal $I$ such that $M/I\cong \Gamma(H,u)$, and there is an injective mapping $\delta_I: \Gamma(H,u)\cong M/I \to M$ satisfying the conditions of Definition \ref{de:wlexid}; we set $M_I:= \delta_I(M/I)$.

We define the category $\mathcal {WLP}_s\mathcal{MV}_{H,u}$ of weakly $(H,u)$-lexicographic pseudo MV-algebras whose objects are quadruplets $(M,I,M_I,\delta_I)$, where $M$ is a weakly $(H,u)$-perfect pseudo MV-algebra with a weakly lexicographic ideal $I$ such that $M/I \cong \Gamma(H,u)$, $\delta_I:\Gamma(H,u)\cong M/I \to M_I$, see Definition \ref{de:wlexid}. If $M=\Gamma(H\lex G,(u,b))$, we set $J=\{(0,g): g \in G^+\}$, $M_J=\{(t,0): t \in [0,u]_H\}$, $\delta_J(t)=(t,0)$, $t \in [0,u]_H$.

If $(M_1,I_1,M_{I_1},\delta_{I_1})$ and $(M_2,I_2,M_{I_2},\delta_{I_2})$ are two objects of $\mathcal {WLP}_s\mathcal{MV}_{H,u}$, then a morphism $f: (M_1,I_1,M_{I_1},\delta_{I_1}) \to (M_2,I_2,M_{I_2},\delta_{I_2})$ is a homomorphism of pseudo MV-algebras $f: M_1 \to M_2$ such that

$$ f(I_1) \subseteq I_2,\quad f(M_{I_1})\subseteq M_{I_2}, \ \mbox{ and }
f\circ \delta_{I_1}=\delta_{I_2}.
$$

Now we define the category of pointed $\ell$-groups, $\mathcal{PLG}$, i.e., the objects are couples $(G,b)$, where $G$ is an $\ell$-group and $b\in G^+$ is a fixed element, and morphisms are homomorphisms of $\ell$-groups preserving fixed elements. We note that the class of pointed $\ell$-groups is a variety whereas the class of unital $\ell$-groups not.

Define a mapping $\mathcal M^w_{H,u}: \mathcal{PLG} \to  \mathcal {LP}_s\mathcal{MV}_{H,u}$ as follows: for $(G,b)\in \mathcal{PLG},$ let
$$
\mathcal M^w_{H,u}(G,b):= (\Gamma(H\lex G,(u,b)),J,H_J,\delta_J)
$$
and if $h: (G,b) \to (G_1,b_1)$ is an $\ell$-group homomorphism preserving fixed elements, then

$$
\mathcal M^w_{H,u}(h)(t,g)= (t, h(g)), \quad (t,g) \in \Gamma(H\lex G,(u,b)). 
$$

Finally, we present a categorical equivalence in the same way as it was proved in the previous section.

\begin{theorem}\label{th:5.4}
The functor ${\mathcal  M}^w_{H,u}$ defines a categorical
equivalence of the category $\mathcal{PLG}$ and the
category $\mathcal {WLP}_s\mathcal{MV}_{H,u}$ of weakly $(H,u)$-lexicographic pseudo MV-algebras.
\end{theorem}

\begin{problem}
{\rm We note that the class of $(H,u)$-perfect pseudo MV-algebras does not form a variety. However, it would be interesting to know an equational basis of the variety of pseudo MV-algebras generated by the class of $(H,u)$-lexicographic (or weakly $(H,u)$-lexicographic) pseudo MV-algebras. For example, if $(H,u)=(\mathbb Z,1)$, the basis is $2.x^2=(2.x)^2$, see \cite[Rem 5.6]{DDT}.
}
\end{problem}

\section{Conclusion}

We have exhibited conditions when a pseudo MV-algebra $M$ can be represented as an interval in the lexicographic product of a fixed linearly ordered unital group $(H,u)$ with an $\ell$-group $G$, both groups are not necessarily Abelian. A crucial condition was the existence of a lexicographic normal ideal $I$ such that $M/I \cong \Gamma(H,u)$, or equivalently, $M$ is a strongly $(H,u)$-perfect pseudo MV-algebra (i.e. $M$ can be decomposed into a system of comparable slices indexed by elements of the interval $[0,u]_H$). Such algebras have a representation $M \cong \Gamma(H\lex G,(u,0))$, we called them also $(H,u)$-lexicographic pseudo MV-algebras, Theorem \ref{th:local1} and Theorem \ref{th:3.8}, and they are not semisimple. We have shown that the category of $(H,u)$-lexicographic pseudo MV-algebras is categorically equivalent to the category of $\ell$-groups, Theorem \ref{th:4.4}.

In addition, we have also studied conditions when a pseudo MV-algebra can be represented in the form $\Gamma(H\lex G,(u,b))$, where $b\in G^+$ is not necessarily the zero element. We call such algebras weakly $(H,u)$-lexicographic, or weakly $(H,u)$-perfect pseudo MV-algebras. A fundamental notion was a weak lexicographic ideal. Their representation is given in Theorem \ref{th:weak}. This category is categorically equivalent to the category of pointed $\ell$-groups whose objects are pairs $(G,b)$ where $b\in G^+$ is a fixed element, Theorem \ref{th:5.4}.

We hope that this research will inspired an additional study of MV-algebras and pseudo MV-algebras that are not semisimple.

\end{document}